\documentclass[a4paper,english, nostix]{amsart}
\usepackage[latin1]{inputenc}
\usepackage{amsmath,babel,amsthm}

\newtheorem{theorem}{Theorem}[section]
\newtheorem{corollary}[theorem]{Corollary}
\newtheorem{proposition}[theorem]{Proposition}
\newtheorem{lemma}[theorem]{Lemma}
\newtheorem*{remark}{Remark}\newcommand{\beq}{\begin{equation}}
\newcommand{\eeq}{\end{equation}}

\newcommand{\mbfu}{\mathbf{u}}
\newcommand{\mbfx}{\mathbf{x}}
\newcommand{\mbfe}{\mathbf{e}}
\newcommand{\C}{\mathbb{C}}
\newcommand{\Z}{\mathbb Z}
\newcommand{\R}{\mathbb{R}}
\def\cP{{\mathcal P}}
\begin{document}

 \title[Discrete Evolutions]{ Sharp Uniqueness Results for Discrete Evolutions}
\author{Yurii Lyubarskii } 
\address{Department of Mathematical Sciences, Norwegian University of Science  and Technology, Trondheim, 7491, Norway}
\email{yura@math.ntnu.nu}
\author{ Eugenia Malinnikova}
\address{Department of Mathematical Sciences, Norwegian University of Science and Technology, Trondheim, 7491, Norway}
\email{eugenia@math.ntnu.nu}
{\it To Helge Holden on the occasion of his 60th birthday}

.
 
\maketitle
\begin{abstract}
We prove sharp uniqueness results for a wide class of one-dimensional discrete evolutions. The proof is based on a construction from  the theory of  complex Jacobi matrices  combined with growth estimates of entire functions.
\end{abstract}


\section{Introduction}
\let\thefootnote\relax\footnote{The research was supported by Grant 213638 of the Research Council of Norway}
We study solutions of discrete  evolution equations of the form
\begin{equation}
\label{eq:ev}
\partial_t \mbfu=A\mbfu,
\end{equation}
where  
  $\mbfu:[0,T]\to l^2(X)$  for some Hilbert space $X$, $\mbfu=\{u_k\}_k,\ u_k:[0,T]\to X$, and  $A$ is a bounded operator on $l^2(X)$ of a special form. Namely, we assume that the matrix of $A$
 (its elements are operators in $X$) is banded, i.e. contains just a finite number of non-zero diagonals. 
  
  We are looking for uniqueness result of the following type:
  
  \medskip
  
  {\em If a solution $\mbfu= \{u_k\}_k$  of \eqref{eq:ev} decays sufficiently fast in spatial variable $k$ at two moments of time $t=0,T$,  then   $\mbfu\equiv 0$.  }
  
  \medskip

  The model example of such evolution is the discrete Schr{\"o}ding\-er equation 
	$\partial_t \mbfu= -i(\Delta_d + V)\mbfu$ 
	on the standard lattice $\Z^d$. For 
 this case  we set $X=l^2(\Z^{d-1})$, i.e. the space $l^2(\Z^d)$ is considered as $l^2(l^2(\Z^{d-1}))$ and the discrete Laplace operator  on $d$-dimensional lattice, $\delta_d:l^2(\Z^d)\rightarrow l^2(\Z^d)$ is defined inductively,
\begin{eqnarray}\label{eq:dL}
(\Delta_1\mbfu )_k&=&u_{k+1}+u_{k-1}-2u_k\quad {\text{for}}\ \mbfu=\{u_k\}\in l^2(\Z)\quad {\text{and}}\\ \nonumber
(\Delta_d\mbfu)_k&=&u_{k+1}+u_{k-1}-2u_k+\Delta_{d-1}u_k \quad {\text{for}}\ \mbfu=\{u_k\}\in l^2(l^2(\Z^{d-1})).
\end{eqnarray} 
Further, the potential part is  $(V\mbfu)_k=V_ku_k$, $V=\{V_k\}$,  where $V_k:l^2(\Z^{d-1})\to l^2(\Z^{d-1})$ are diagonal operators for $k\in\Z$. The uniqueness problem for this evolution has been considered in \cite{JLMP, FB1, FB, FBV, AT}.

Our  research is motivated by  a remarkable series of papers
  \cite{CEKPV, EKPV1, EKPV2} (see also references therein) which studied the continuous case. In these articles   a sharp uniqueness statement   is obtained  for solutions of
   Schr\"odinger equations with time-dependent  potentials, the result is applicable to  some 
   non-linear equations. For the potential-free Schr\"odinger evolution the uniqueness statement can be considered as a version of the classical Hardy uncertainty principle. 

 The Fourier transform applied to both the discrete and continuous Schr\"o\-din\-ger evolutions transforms the uniqueness questions into those on growth of analytic functions. In \cite{JLMP} and \cite{FB1}   the theory of entire functions has been applied to  the model case of free  discrete evolution ($A=-i\Delta_d$).
  It was proved that in dimension $d=1$ the inequality
\[
|u_n(0)|+|u_n(1)|<\frac{1}{\sqrt{|n|}}\left(\frac{e}{2|n|}\right)^{|n|},\quad n\in\Z\setminus\{0\},
\]
implies $u_n(t)=Ai^{-n} e^{-2it}J_n(1 - 2t)$, where $J_n$ is the Bessel function. In particular a solution to the free Schro\"odinger evolution equation cannot decay faster than $J_n(1)$ simultaneously at $t=0$ and $t=1$. This result was also generalized to special classes of time-independent potentials, first those with compact supports \cite{JLMP} and then fast decaying  \cite{AT}. General bounded potentials were considered in \cite{JLMP} (in dimension $d=1$) and \cite{FBV} (in arbitrary dimension).  For time-dependent potentials the  uniqueness results obtained in \cite{JLMP, FBV} show that the inequality
\[|u(t,k)|\le C\exp(-\gamma|k|\log|k|)\] 
for some fixed $\gamma>\gamma_0$ implies $u\equiv 0$, however these results are not sharp.

In this note we  combine the entire function techniques developed in  \cite{JLMP} with some ideas  from the theory of complex  Jacobi matrices in order to consider  general discrete models with time-independent banded operator $A$. Thus we cover for example one-dimensional heat and Schr\"odinger evolutions with bounded potentials as well as some discrete versions of higher order  one-dimensional operators and also some higher dimensional operators (with very specific potentials).

The article is organized as follows. The next section contains preliminaries related to banded operators and generalized eigenvectors. We also consider some model examples of operator $A$ where the problem \eqref{eq:ev}
admits explicit solution.  In section 3  we apply the theory of entire functions to show that any solution to  general time-independent evolution which  decays sufficiently fast at two times is orthogonal to all generalized eigenvectors of the adjoint operator $A^*$, this argument holds for general banded operators on $l^2(X)$.  
For the case of a selfadjoint operator $A$ and $X=\C$ one can apply general results on completeness of the set of generalized eigenvectors in order to see that  this orthogonality implies that the solution is trivial. At the end of section 3 the multidimensional selfadjoint case, i.e. when $A=A^*$ and $X=l^2(Z^{d-1})$, is also considered. We demand additional decay of solution in complimentary spatial variables. This decay is needed to include the space $l^2(\Z^d)$ in a Gelfand triple and apply a general result on the completeness of the set of generalized eigenvectors. The more complicated non-selfadjoint case is  presented in Section 4.  The construction is inspired by a version of Shohat--Favard theorem for complex Jacobi matrices. We consider first the case $X=\C$ in order to show the main ideas without further technical details. For general $X$ we need an additional assumption. Namely we assume that the matrix entries of the operator $A$ commute with each other. We don't know if this assumption is necessary. 
 In Section 5 we consider a 
closely related question on decay of the solutions of the discrete stationary equation.

\subsection*{Acknowledgment} This work has been done while the authors were visiting Department of Mathematics at Purdue University. It is our pleasure to thank the department for hospitality. We also want to thank A. Pushnitski for a useful discussion.

\section{Preliminaries}

\subsection{Banded operators}
We consider operators $A:l^2(X)\to l^2(X)$, where $X$ is a Hilbert space, 
\[
l^2(X)=\left\{\mbfx =\{x_j\}_{j\in\Z},\  x_j\in X,\  \|\mbfx \|^2=\sum_j\|x_j\|_X^2<\infty\right\}.
\] This includes operators on $l^2$ sequences over $\Z^d$, we identify this space with $l^2(l^2(\Z^{d-1}))$. We assume that  $A:l^2(X)\to l^2(X)$ is a banded operator,
i.e., for some integer $s$
\begin{equation}
\label{eq:A}
(A\mathbf{x})_j=\sum_{k=j-s}^{j+s} A_{j,k}x_k, \ \mbfx\in l^2 (X),
\end{equation}
where  $A_{j,k}:X\to X$ are bounded operators.  We will refer to these operators as to entries of $A$. The number $2s$ plays the role of order  of $A$,  it will define 
the order of decay in the corresponding uniqueness statement.

In addition we assume that the "external" entries  $A_{j, j\pm s}$ are invertible and
\begin{equation}
\label{eq:matrix}
\|A_{j,j\pm s}^{-1}\|\le \delta^{-1},\ \|A_{j,k}\|\le a,
\end{equation}
for some $a,\delta >0$, independent of $j$.
 
 Clearly, the adjoint operator $A^*$ is also banded and satisfies the same conditions \eqref{eq:matrix}.
 
 \subsection{Generalized eigenvectors}
We consider generalised eigenvectors of $A^*$.  Since $A^*$ is a banded operator, the expression  $A^*\mathbf{e}$ makes sense for any sequence  $\mathbf{e}=\{e_j\}_{j\in\Z}$ with $e_j\in X$. 
 We say that $\mathbf{e} $   is  a generalized eigenvector if $A^*\mathbf{e}=\lambda_0\mathbf{e}$ for some $\lambda_0\in\C$.

For any $\lambda\in\C$ and any vectors $e_{-s}, e_{-s-1},..., e_{s-1}\in X$ there exists  a unique vector  $\mathbf{e}(\lambda)=\{e_j(\lambda)\}_{j\in\Z}$ with $e_j(\lambda)\in  X$ such that \[e_{j}(\lambda)=e_j,\  j=-s,...,s-1,\quad{\text{and}}\quad A^*\mathbf{e}(\lambda)=\lambda \mathbf{e}(\lambda).\] It is defined by
\begin{align}
\label{eq:ev1} e_j(\lambda)&=e_j,\quad j=-s,..., s-1,\\ 
\label{eq:ev2} e_{s+k}(\lambda)&=(A^*_{s+k,k})^{-1}\left(\sum_{m=-s}^{s-1}A^*_{m+k,k}e_{m+k}(\lambda)-\lambda e_{k}(\lambda)\right),\quad k\ge 0,\\
\label{eq:ev3} e_{-s-k}(\lambda)&=(A^*_{-s-k,-k})^{-1}\left(\sum_{m=-s+1}^sA^*_{m-k,-k}e_{m-k}(\lambda)-\lambda e_{-k}(\lambda)\right),\quad k\ge 1.
\end{align}
The vectors $e_j(\lambda)$ are polynomials in $\lambda$ (with values in $X$) of degree less than $[|j|/s]+1$. 
 Let $M=\max_{-s\le j<s}\|e_j\|$,  then an induction argument yields
\[
\| e_n(\lambda)\|\le My^{n+s},\  n\ge -s,\]
for all $y>1$  such that $y^{2s}\ge \delta^{-1}(a(y^{2s-1}+y^{2s-2}+...+y+1)+|\lambda|y^{s})$. We multiply the last inequality by $(y-1)$ and see that it holds if \[y^{2s+1}\ge (a\delta^{-1}+1)y^{2s}+\delta^{-1}|\lambda|y^{s+1}.\] Which is in turn satisfied if we choose $y\ge \delta^{-1/s}|\lambda|^{1/s}+a\delta^{-1}+1$. Similar estimates can be repeated for negative $n$. We obtain
\begin{equation}\label{eq:evest}
\|e_{ks+r}(\lambda)\|,\|e_{-ks-r-1}(\lambda)\|\le CM\delta^{-k}(|\lambda|+b)^{k+2},\quad k\ge 1,\  0< r\le s,\end{equation}
for some $b=b(s,a,\delta).$

\subsection{Model examples} Our main example is $A=\alpha\Delta_d$, where  $\Delta_d$  is  the  discrete lattice Laplacian given by \eqref{eq:dL} and $\alpha\in\C$.
 Clearly, this is an operator of the form \eqref{eq:A} with  $X=l^2(\Z^{d-1})$, $s=1$, $A_{j,j\pm 1}=\alpha I$ and  $A_{j,j}=\alpha(\Delta_{d-1}-2I)$.
 
 For $d=1$ solutions to the corresponding evolution problem can be expressed in terms of the  Bessel functions of the second  kind, one of them is    
\[
u_n(t)=I_{n}(2\alpha (t-t_0))e^{-2\alpha (t-t_0)}. 
\]

In higher dimension we have solutions of the form
\[ 
u_n(t)=\left\{I_n(2\alpha(t-t_0))\left(\prod_{l=1}^{d-1} I_{n_l}(2\alpha(t-t_0)\right)e^{-2d\alpha(t-t_0)}\right\}_{(n_1,...n_{d-1})\in\Z^{d-1}}.
\]

The powers of the discrete Laplacian provide examples of  higher order operators that satisfies our assumptions. However a simpler model is given by the  operator with  $A_{j,j\pm s}= I$, $A_{j,j}=-2 I$ and $A_{j,k}=0$ otherwise. Then a  solution is given by
\[
u_n(t)=C_rI_{q}(2(t-t_0)),\quad n=qs+r,\  0\le r<s.
\]
For $t_0=T/2$  this solution indicates the critical speed of decay in spatial variables:
$$
|u_n(0)|+|u_n(T)| \asymp |q|^{-1/2} \left ( \frac{eT}{2|q|}\right)^{|q|}.
$$

\section[Orthogonality to generalized eigenfunctions]{Orthogonality to generalized eigenfunctions and self-adjoint operators} 
\subsection{Controlled decay} We  need the following auxiliary statement.
\begin{lemma}
\label{lemma:pre} Suppose that  $\mbfu:[0,T]\rightarrow l^2(X)$ is a solution to 
\eqref{eq:ev} and  $A$ satisfies conditions \eqref{eq:A} and \eqref{eq:matrix}. 
  Suppose further that
\begin{equation}\label{eq:gr0}
\|u_j(0)\|_X\le C_0^{k}k^{-k/2},\quad k=[|j|/s]+1.
\end{equation}
Then for each $t\in[0,T]$ there exists $C_t$ such that
\begin{equation}
\label{eq:gr2}
\|u_j(t)\|_X\le C_t^{k}k^{-k/2},\quad k=[|j|/s]+1,\quad  t\in[0,T]. 
\end{equation}
\end{lemma}
\begin{proof}
Consider the function  $f_B(t)=\sum_{j}B^{|j|}\|u_j(t)\|_X^2$. It satisfies the differential inequality $f'_B(t)\le C_1B^s f_B(t)$, where $C_1$ does not depend on $B$.
Therefore
\beq
\label{eq:diffin}
f_B(t)\le e^{C_1B^s t} f_B(0).
\eeq
 In addition, \eqref{eq:gr0} implies that $f_B(0)\le e^{C_2B^s}$. Then $f_B(t)\le e^{C_3B^s}$ with $C_3=C_3(t)$ and,  in particular, $\|u(j,t)\|^2\le B^{-|j|}e^{C_3B^s}.$ We
 optimize the last inequality by choosing $B\asymp k$ and get the required estimate \eqref{eq:gr2}.
 
 In this argument we assumed that $f_B(t)$ is well-defined for all $B$. To justify this one can first consider the functions
\[
\tilde {f}_{N,B}(t)=\sum_{j}\min\{B^{|j|}, B^N\}\|u(j,t)\|_X^2,
\]
  obtain  estimate \eqref{eq:diffin} for these functions with constants independent of $N$, and then pass to the limit as $N\to\infty$.
\end{proof}
\begin{corollary} Let the function $\mbfu:[0,T]\rightarrow l^2(X)$ satisfy the hypothesis of  Lemma \ref{lemma:pre} and  $\mbfe$ be a generalized eigenvector of $A^*$. Then the inner product 
$$
 \langle \mbfu(t),\mathbf{e}\rangle  = \sum_{j\in\Z}\langle u_j(0), e_j\rangle_X
$$
is well-defined.
\end{corollary}

This statement follows from the  lemma  and the fact that $\|e_j\|$ grows in $j$ not faster than exponentially, see \eqref{eq:evest}.


\subsection{Orthogonality}
We  now prove that   any  solution to \eqref{eq:ev} which decays at two moments    faster
 than the model one is orthogonal to all generalized eigenvectors of $A^*$.
\begin{proposition} \label{pr:onesa}
Suppose that $A:l^2(X)\to l^2(X)$ is a banded operator satisfying \eqref{eq:A} and \eqref{eq:matrix}. Suppose that $\mathbf{e}$ is a generalized eigenvector of $A^*$.
Let further $\mbfu:[0,T]\to l^2(X)$ satisfy $\partial_t \mbfu=A\mbfu$,  and
\begin{equation}\label{eq:gr1}
\|u_j(t)\|_X\le Ce^{|k|}(2+\varepsilon)^{-|k|}|k|^{-|k|}T^{|k|}\delta^{|k|},\quad k=[j/s],\quad {\text{when}}\ \ t=0,T.
\end{equation}
Then $\langle u(0),\mathbf{e}\rangle=0 $.  
\end{proposition}
\begin{proof}  Let  $A^*\mathbf{e}=\lambda_0 \mathbf{e}$, $\mathbf{e}=\{e_j\}_j$, we define a family $\mathbf{e}(\lambda)$ 
of generalized eigenvectors by (\ref{eq:ev1}-\ref{eq:ev3}).  
In this way the eigenvector $\mathbf{e}$ is included into an analytic family of eigenvectors $\mathbf{e}(\lambda)$, $\lambda\in\C$. We consider the family of entire functions  
$$
\phi(t,\lambda)=\langle  \mathbf{e}(\lambda), \mbfu(t,)\rangle_{l^2(X)}=\sum_j \langle e_j(\lambda), u_j(t)\rangle_{X}.
$$
Differentiating with respect to $t$, we obtain
\[
\partial_t\phi(t,\lambda)=\langle \mathbf{e}(\lambda), A\mathbf{u}\rangle=\langle A^*\mathbf{e}(\lambda), \mathbf{u}\rangle =\lambda\phi(t, \lambda).
\]
Then for each $\lambda$ we have 
\beq
\label{eq:entire}
\phi(t,\lambda)=e^{\lambda t}\phi(0,\lambda).
\eeq
At the same time estimates (\ref{eq:gr1}) and (\ref{eq:evest}) give
\beq
\label{eq:growth}
|\phi(0,\lambda)|, |\phi(T,\lambda)|\le C e^{T|\lambda|/(2+\varepsilon)}.
\eeq
  
 The proof can be now completed in the same spirit as Theorem 2.3 in   \cite{JLMP}. We include a brief argument  in order to make 
 the presentation mainly self-contained and refer the reader to  monograph \cite{L} for definitions and basic facts related to entire functions. Let 
 $$
 h_0(\theta)=\limsup_{r\to \infty}\frac{\ln |\phi(0,re^{i\theta})|}{r}, \  h_T(\theta)=\limsup_{r\to \infty}\frac{\ln |\phi(T,re^{i\theta})|}{r},  \ 
 \theta\in[0,2\pi]
 $$
  be the indicator functions of the entire functions $\phi(0,\lambda)$ and $\phi(T,\lambda)$.  Relation \eqref{eq:entire} for $\theta=0$ and $t=T$ yields
  \beq
  \label{eq:first}
  h_T(0)=T+h_0(0).
  \eeq
  On the other hand it follows from \eqref{eq:growth}  that 
  $$
  h_0(\theta), h_T(\theta) < \frac{T}{2+\varepsilon}, \  \theta\in [0,2\pi],
  $$
  and, by (5) in  \cite[Lecture 8]{L} (for our case $\rho=1$ in this relation),
   $$
  |h_0(\theta)|, |h_T(\theta)| < \frac{T}{2+\varepsilon}, \ \theta\in [0,2\pi].
  $$
  The later inequality is incompatible with \eqref{eq:first} unless $\phi(0,\lambda)=0$.
\end{proof}

\subsection{Selfadjoint case} In this subsection $X=l^2(\Z^{d-1})$ and $A=A^*$ or $A=cA^*$ for some $c\in\C$. This happens for example in the model cases of heat or Schr\"odinger evolutions with real potentials.

The elements in $l^2(\Z^d)$ are denoted by $\mathbf{x}=\{x_k\}_k,\ x_k\in l^2(\Z^{d-1})$. We say that $k$ is the main variable and call the $d-1$ arguments of  $x_k$ complementary spatial variables. In order to obtain the completeness of the generalized eigenvectors, and thus prove the uniqueness theorem applying the results of the previous subsections, we include $l^2(\Z^d)$ into an appropriate Gelfand triple $\Phi\hookrightarrow l^2(\Z^d)\hookrightarrow \Phi'$, see e.g. \cite{B, GS, GV}. This can be done by demanding some decay of solution in complementary variables.

Given $\alpha\in\R$ we consider the weighted space
\[l^2_\alpha(\Z^{d-1})=\{\mathbf{c}=\{c_m\}_{m\in\Z^{d-1}}:\|\mathbf{c}\|_\alpha^2=\sum_{m\in\Z^{d-1}}(1+|m|)^\alpha|c_m|^2<\infty\}.\]
\begin{theorem}\label{th:sa}
Suppose that $\alpha>d-1$ and $A:l^2(l^2(\Z^{d-1}))\to l^2(l^2(\Z^{d-1}))$, $(A\mbfu)_j=\sum_{k=j-s}^{j+s}A_{j,k}u_k$,  is a banded operator, where $A_{j,k}$ are bounded in $l^2(\Z^{d-1})$ as well as in $l^2_\alpha(\Z^{d-1})$. Let further the external operators $A_{j,j\pm s}$ be invertible in $l^2_\alpha(\Z^{d-1})$ and
\[ \|A_{j,j\pm s}^{-1}\|_{l^2_\alpha\to l^2_\alpha}\le \delta^{-1},\ \|A_{j,k}\|_{l^2_\alpha\to l^2_\alpha}\le M,\ k=j-s,...,j+s.\]
If  $\mbfu:[0,T]\to l^2(l^2_\alpha(\Z^{d-1}))$ satisfies $\partial_t \mbfu=A\mbfu $, and the decay condition  in main spatial variable
\begin{equation*}
\|u(t,j)\|_{l^2_\alpha(\Z^{d-1})}\le Ce^{|k|}(2+\varepsilon)^{-|k|}|k|^{-|k|}T^{|k|}\delta^{|k|},\quad k=[j/s],\quad {\text{for}}\ \ t=0,T,
\end{equation*}
Then $u\equiv 0$
\end{theorem}
\begin{remark} In the model case, when  $A$ is a  the sum of the Laplace operator and a real bounded potential (up to a unimodular factor), the operators $A_{j,k}$ are bandlimited themselves and bounded in weighted spaces, moreover $A_{j,j\pm s}$ are identity operators and the norm estimate holds with $\delta=1$.
\end{remark} 
\begin{proof} We consider the space 
\[\Phi=\{\mathbf{C}=\{\mathbf{c}_k\}_{k\in\Z},\ \mathbf{c_k}\in l^2_\alpha(\Z^{d-1}): \|\mathbf{C}\|_\Phi^2=\sum_{k\in\Z}e^{|k|^{1/2}}\|\mathbf{c_k}\|^2_{\alpha}<\infty\}.\]
Then the dual space (with respect to pairing in $l^2(\Z^d)$ is
\[\Phi'=\{\mathbf{C}=\{\mathbf{c}_k\}_{k\in\Z},\ \mathbf{c_k}\in l^2_\alpha(\Z^{d-1}): \|\mathbf{C}\|_{\Phi'}^2=\sum_{k\in\Z}e^{-|k|^{1/2}}\|\mathbf{c_k}\|^2_{{-\alpha}}<\infty\}.\]
We have $\Phi\hookrightarrow l^2(\Z^d)\hookrightarrow\Phi'$ and the inclusion is a Hilbert-Schmidt operator since $\alpha>d-1$. We observe also that $A:\Phi\to\Phi$ and hence $A:\Phi'\to\Phi'$ are bounded operators. By repeating the arguments of the previous section, we obtain that $\mathbf{u}(0)\in\Phi$ is orthogonal to all generalized eigenvectors of $A$ in $\Phi'$. Then by general result, see for example \cite[Chapter V,Theorem 1.4]{B}, we obtain that $\mathbf{u}(0)=0$. 
\end{proof}
\section{A sharp uniqueness result for bounded evolutions}
\label{s:CJ}
\subsection{Main result} We are now ready to prove our main result.
\begin{theorem}\label{th:m}
Suppose that $A:l^2(X)\to l^2(X)$, $(A\mbfu)_j=\sum_{k=j-s}^{j+s}A_{j,k}u_k$,  is a banded operator satisfying \eqref{eq:A} and \eqref{eq:matrix}. Further, assume that all operators $A_{j,k}$ commute.
Let  $\mbfu:[0,T]\to l^2(X)$ satisfy $\partial_t \mbfu=A\mbfu $, and the decay condition \eqref{eq:gr1}:
\begin{equation*}
|u(t,j)|\le Ce^{|k|}(2+\varepsilon)^{-|k|}|k|^{-|k|}T^{|k|}\delta^{|k|},\quad k=[j/s],\quad {\text{for}}\ \ t=0,T,
\end{equation*}
Then $u\equiv 0$
\end{theorem}
The theorem follows from Proposition \ref{pr:onesa} and the proposition below. In dimension one our result can be applied to  both heat and Schr\"odinger evolutions with bounded time-independent potentials as well as to evolutions defined by higher order difference operators. In higher dimension this approach allows to work only with  potentials depending on the variable in the direction of decay.
\begin{proposition}
\label{pr:ortog}
Let $\mathbf{u}=\{u_j\}_{j\in\Z}\in l^2(X)$ be such that 
\[
\sum_{j\in\Z}C^{|j|}\|u_j\|<\infty
\] 
for every $C$. Let also $\langle \mathbf{e},\mathbf{u}\rangle =0$ for each generalized eigenvector $\mathbf{e}$ of a banded operator $A^*$. Then $\mathbf{u}=\mathbf{0}$.
\end{proposition}
Our proof of the above proposition is inspired by a  well known construction, sometimes referred to as the Shohat-Favard theorem for complex Jacobi matrices. We refer the reader to the survey articles \cite{Bea, Bec} and references therein. 
\subsection{Dimension one}
To avoid extra technical details and explain the idea we first assume that $X=\C$ and write $A_{j,k}=a_{j,k}\in\C$

\begin{proof}[Proof of Proposition \ref{pr:ortog}, $X=\C$] Consider the   families of polynomials
 \[
 P_j^{(r)}(\lambda),\quad r=-s, -s+1, ..., 0, ...,  s-1,\ j\in\Z
 \] 
 defined by the relations 
 \begin{gather}
 P_j^{(r)}(\lambda)=\delta_{j,r}, \quad   j=-s, -s+1, ..., 0, ...,  s-1,  \notag \\
 \label{eq:induction}
  \lambda P_j^{(r)}(\lambda)=\sum_{k=j-s}^{j+s}\bar{a}_{k,j}P_k^{(r)}(\lambda).
  \end{gather}
For each $\lambda\in\C$ and $r=-s,...,s-1$ the vector $\mathbf{v}^{(r)}(\lambda=\{P_j^{(r)}(\lambda)\}_{j}$is a generalized eigenvector of $A^*$ with eigenvalue $\bar{\lambda}$.
Therefore 
\beq
\label{eq:scalar}
\sum_j u_j\overline{P_j^{(r)}(\lambda)}=0.
 \eeq

Let $\bar{A}:l^2(\C)\to l^2(\C)$ denote the "complex conjugate" of $A$:
\[
(\bar{A}\mbfu)_j=\sum_{k=j-s}^{j+s}\bar{a}_{j,k}u_k.
\]
 We consider $P^{(r)}_n(\bar{A}): l^2(\Z)\to l^2(\Z)$. 
The scalar relation \eqref{eq:induction} now yields
 \[
 \bar{A}P_j^{(r)}(\bar{A})=\sum_{k=j-s}^{j+s}\bar{a}_{k,j}P_k^{(r)}(\bar{A}).
 \]
 This in particular implies that 
\begin{equation}
\label{eq:evesta}
\|P^{(r)}_{n}(\bar{A})\|\leq C^{|n|} \ \mbox{for some} \ C>0. 
\end{equation}
similar to (\ref{eq:evest}).

We claim that \eqref{eq:scalar} implies 
\[\sum_{n} u_n\overline{P_n^{(r)}(\bar{A})}=0\]
and due to \eqref{eq:evesta} the series converges absolutely.

Let further ${\bf{\sigma}}^{(n)}$ be the $n$-th coordinate vector in $l^2(\Z)$. An induction argument shows that
 \[
 \sum_{r=-s}^{s-1} P^{(r)}_{n}(\bar{A})\mathbf{\sigma}^{(r)}=\mathbf{\sigma}^{(n)}.
 \]
 Then
\[
0=\sum_{r=-s}^{s-1}\sum_{n}u_n\overline{P_n^{(r)}}(\bar{A}){\mathbf{\sigma}}^{(r)}=\sum_nu_n\sigma^{(n)}.
\]
Hence $u\equiv 0$.
\end{proof}

\subsection{General case}
We extend  the above construction to banded operators on $l^2(X)$ with commuting entries. 

\begin{proof}[Proof of Proposition \ref{pr:ortog}, General case]
 We split the proof into several steps.
\subsubsection*{Step 1} We define families of operator-polynomials $\{P_j^{(r)}(\lambda)\}_j$, $-s\le r <s$, $\lambda\in\C$ by
\begin{gather}
P_r^{(r)}=I, \ P_j^{(r)}=0, \ j\neq r  \ \mbox{and} \  -s\le k<s  \notag  \\
\label{eq:polynomial}
\lambda P_j^{(r)}(\lambda)=\sum_{k=j-s}^{j+s}A_{k,j}^*P_k^{(r)}(\lambda).
\end{gather}

For any $x\in X$ the sequence $\mathbf{v}=\{v_j\}_j=\{P^{(r)}_j(\lambda)x\}_j$ is a generalized eigenvector of $A^*$, $A^*\mathbf{v}=\lambda \mathbf{v}$.

We have  $P_j^{(r)}(\lambda)=\sum_{m\ge 0} \lambda^m C_{j,m}^{(r)}$, where $C_{j,m}^{(r)}:X\to X$ and the sum is finite. Moreover, all coefficients $C_{j,m}^{(r)}$ are products of the operators  $A_{k,l}^*$ and their inverses (we will use this fact to interchange the order of operators).

Now the orthogonality relation $\mathbf{u}\perp \{P_j^{(r)}(\lambda)x\}_j$ implies
\[
0=\sum_j\langle u_j,P^{(r)}_j(\lambda)x\rangle_X=\sum_{m}\lambda^m \sum_j\langle u_j, C_{j,m}^{(r)}x\rangle_X,
\]
the series converges since we assume that $\|u_j\|_X$ decays fast in $j$. We conclude that each coefficient $\sum_j\langle u_j, C_{j,m}^{(r)}x\rangle_X$ vanishes. Then 
\begin{equation}
\label{eq:mz}
\sum_j (C_{j,m}^{(r)})^*u_j=0.
\end{equation}

\subsubsection*{Step 2}
Denote by $\bar{A}$ the "conjugate"operator 
\[
\bar{A}\mathbf{v}=\bar{A}\{v_j\}=\{(\bar{A}\mathbf{v})_j\},\quad 
(\bar{A}\mathbf{v})_j=\sum_{k=j-s}^{j+s} A^*_{j,k}v_k.
\]
By $i_m$ we denote the embedding  $X\hookrightarrow l^2(X)$ that places a given vector $x\in X$ into $m$-th position and zeros in all other positions:
\[
(i_m x)_k=\delta_{m,k}x.
\]
Define further
\begin{equation}\label{eq:mrec}
\cP_j^{(r)}u=\sum_{m\ge 0} \bar{A}^m i_rC_{j,m}^{(r)}u,\quad u\in X,\quad \cP_j^{(r)}:X\to l^2(X).
\end{equation}
Then (\ref{eq:mrec}), \eqref{eq:polynomial} and the commutation relation $A_{k,j}^*C_{k,m}^{(r)}=C_{k,m}^{(r)}A_{k,j}^*$ imply
\begin{equation*}
\bar{A}\cP_j^{(r)}u=\sum_{k=j-s}^{j+s}\cP_{k}^{(r)}A^*_{k,j}u.
\end{equation*}

We show by induction that for any $v\in X$
\begin{equation}\label{eq:mid}
\sum_{r=-s}^{s-1}\cP_n^{(r)}v=i_nv.
\end{equation}
Indeed, for $n=-s,..,s-1$ this follows from the definition of   $\cP_n^{(r)}$. Further by the recurrence formula
\[\cP_{n}^{(r)}A_{n,n-s}^*v=\bar{A}\cP_{n-s}^{(r)}(v)-\sum_{k=n-2s}^{n-1}\cP_{k}^{(r)}(A_{k,n-s}^*v)\]
Taking the sum with respect to $r$ and using the induction hypothesis, we obtain
\[\sum_{r=-s}^{s-1}\cP_n^{(r)}A_{n,n-s}^*v=\bar{A}i_{n-s}v-\sum_{k=n-2s}^{n-1}i_kA^*_{k,n-s}v=i_n(A^*_{n,n-s}v).\]
Now (\ref{eq:mid}) follows since $A_{n,n-s}^*$ is invertible.

\subsubsection*{Step 3} We denote by $\pi_k$ the $k$th projection of $l^2(X)$ to $X$, $\pi_k\mathbf{v}=v_k.$ Now we fix some $x\in X$ and for each $j\in\Z$ and $r=-s,...,s-1$ 
consider a sequence $\alpha^{(r, j)}=\{\alpha_k^{(r,j)}\}_k\in l^2(\C)$ defined by
\[\alpha^{(r,j)}_k=\langle u_j, \pi_k\cP_j^{(r)}x\rangle_X.\]
Let $\alpha^{(r)}=\sum_j \alpha^{(r, j)}\in l^2$, we have
\[\alpha^{(r)}_k=\sum_j\langle u_j, \pi_k\cP_j^{(r)} x\rangle_X=\sum_m\sum_j\langle u_j, \pi_k\bar{A}^mi_rC^{(r)}_{j,m}x\rangle_X.\]
The coefficients of operators $\bar{A}^m$ are operators from $X$ to $X$, they are products of operators $A_{l,k}^*$. Clearly, $\pi_k\bar{A}^mi_r$ is such a coefficient, it commutes with $C_{j,m}^{(r)}$. Therefore
\[\alpha_k^{(r)}=\sum_m\left\langle \sum_j (C_{j,m}^{(r)})^*u_j,\pi_k\bar{A}^mi_rx\right\rangle_X=0,\]
the last identity follows from (\ref{eq:mz}).

On the other hand, by (\ref{eq:mid})
\[
\sum_{r=-s}^{s-1}\alpha_k^{(r,j)}=\left\langle u_j, \pi_k \left (\sum_r \cP_j^{(r)}x\right )\right\rangle_X=\langle u_j, \pi_ki_j x\rangle_X=
\begin{cases} \langle u_j, x\rangle,\ k=j\\ 0,\ k\neq j
\end{cases}\]
Finally, $0=\sum_r\alpha_k^{(r)}=\sum_j\sum_r\alpha_k^{(r,j)}=\langle u_k,x\rangle$. Thus $u=0$.
\end{proof}
\subsection{Decay of stationary solutions} 
It was mentioned in \cite{FBV} that uniqueness results imply some estimates on the possible decay of stationary solutions of discrete Schr\"odinger operators. We suggest two elementary but reasonably sharp results.
\begin{proposition} Suppose that $A$ is a banded operator on $l^2(X)$ satisfying \eqref{eq:A} and \eqref{eq:matrix}. There exists a constant $c=c(A)$ such that  if a  solution $\mbfu\in l^2(A)$ of $A\mbfu=0$ satisfies
$\|\mbfu_j\|_X\le Ce^{-cj}$
then $u\equiv 0$.
\end{proposition}
\begin{proof}
The recurrence formula implies
\[u_{n-s}(\lambda)=A_{n-s,n}^{-1}\left(\sum_{m=-s+1}^{s}A_{m+n,n}u_{m+n}\right).\]
Clearly, $\|u_{n-s}\|_X\le \delta^{-1}a\sum_{m=-s+1}^s\|u_{n+m}\|_X$. If $M_j=\max_{-s<m\le s}\|u_{j+m}\|_X$ then $M_{j}\ge (2s)^{-1}\delta a^{-1}M_{j-1}$. This actually implies that if \[
\|u(j)\|_X\le Cq^j,\quad q<\delta a^{-1}(2s)^{-1}\]
then $u\equiv 0$.
\end{proof}
We could formulate a bit more genera result, saying that 
\[\liminf_{j\to\infty} \frac{\ln M_j}{j}\ge c.\]
for any non-trivial solution of the stationary equation.

Similar approach can be applied to the case of the discreet Schr\"odinger operator with a bounded potential $V:\Z^d\to \C$, a straightforward computation shows that if  $u:\Z^d\to\C$ satisfies $\Delta_d u+Vu=0$ and
\begin{equation}\label{eq:eigen2}
\liminf_{N\to \infty}\frac{\ln(\max_{|n|_\infty\in\{N,N+1\}}|u(n)|)}{N}<-\|V\|_\infty-4d+1
\end{equation}
where $|n|_\infty=\max\{n_1,...n_d\}$ for $n\in\Z^d$, then $u\equiv 0$.
Indeed, the equation implies
\[
\max_{|n|_\infty=N-1} |u(n)|\le (4d-2+\|V\|_\infty)\max_{|n|_\infty=N}|u(n)|+\max_{|n|_\infty=N+1}|u(n)|,\]
and  (\ref{eq:eigen2}) follows.


\end{document}